\documentclass[a4paper, 10pt]{amsart}
\usepackage{latexsym,amssymb}
\usepackage[english]{babel}
\usepackage{bm}
\usepackage{amsfonts, amsthm,mathtools}
\usepackage{subfig}
\usepackage{amsmath}
\usepackage{mathrsfs}
\usepackage{multirow}
\usepackage{bm}
\usepackage{pgf,tikz}
\usepackage{rotating}
\usepackage{calculator}

\usepackage{xcolor}
\usepackage[margin=3.95cm]{geometry}

\newcommand{\comment}[1]{}

\allowdisplaybreaks

\def\Z{\mathbb{Z}}
\def\N{\mathbb{N}}

\usepackage{amsthm}
\theoremstyle{plain}

\newtheorem*{Theorem*}{Theorem 1}

\newtheorem*{Lemma*}{Lemma}
\newtheorem{Definition}{Definition}[section]
\newtheorem{Remark}[Definition]{Remark}
\newtheorem{Theorem}[Definition]{Theorem}

\newtheorem{Lemma}[Definition]{Lemma}
\newtheorem{Problem}[Definition]{Problem}
\newtheorem{Corollary}[Definition]{Corollary}

\title[Variants of the Erd\H{o}s distinct sums problem and variance]{Variants of the Erd\H{o}s distinct sums problem\\ and variance method}
\author{Simone Costa, Stefano Della Fiore, Andrea Ferraguti}
\begin{document}
\begin{abstract}
Let $\Sigma=\{a_1, \ldots , a_n\}$ be a set of positive integers with $a_1 < \ldots < a_n$ such that all $2^n$ subset sums are pairwise distinct. A famous conjecture of Erd\H{o}s states that $a_n>C\cdot 2^n$ for some constant $C$, while the best result known to date is of the form $a_n>C\cdot 2^n/\sqrt{n}$. In this paper, we propose a generalization of the Erd\H{o}s distinct sum problem that is in the same spirit as those of the Davenport and the Erd\H{o}s-Ginzburg-Ziv constants recently introduced in \cite{CGS} and in \cite{CS}. More precisely, we require that the non-zero evaluations of the $m$-th degree symmetric polynomial are all distinct over the subsequences of $\Sigma$ whose size is at most $\lambda n$, for a given $\lambda\in (0,1]$, considering $\Sigma$ as a sequence in $\mathbb{Z}^k$ with each coordinate of each $a_i$ in $[0,M]$. If $\mathcal{F}_{\lambda,n}$ denotes the family of subsets of $[1,n]$ whose size is at most $\lambda n$, our main result is that, for each $k,m,$ and $\lambda$, there exists an explicit constant $C_{k,m,\lambda}$ such that
$$ M\geq C_{k,m,\lambda} \frac{(1+o(1)) |\mathcal{F}_{\lambda,n}|^{\frac{1}{mk}}}{n^{1 - \frac{1}{2m}}}.$$
\end{abstract}

\maketitle

\section{Introduction}
For any $n\geq 1$, consider sets $\{a_1,\ldots , a_n\}$ of positive integers with $a_1 < \ldots < a_n$ whose subset sums are all distinct. A famous conjecture, due to Paul Erd\H{o}s, is that $a_n \geq C \cdot 2^n$ for some constant $C > 0$. Using the variance method, Erd\H{o}s and Moser \cite{Erdos} (see also \cite{Alon2} and \cite{Guy2}) were able to prove that
$$a_n\geq 1/4\cdot n^{-1/2}\cdot 2^n.$$
No advances have been made so far in removing the term $n^{-1/2}$ from this lower bound, but there have been several improvements on the constant factor, including the
work of Dubroff, Fox, and Xu \cite{Fox}, Guy \cite{Guy}, Elkies \cite{Elkies}, Bae \cite{Bae}, and Aliev \cite{Aliev}. In particular, the best currently known lower bound states that
$$
a_n\geq \sqrt{\frac{2}{\pi}} \frac{1}{\sqrt{n}}2^{n} (1+o(1)).
$$
Two simple proofs of this result, first obtained unpublished by Elkies and Gleason, are presented in \cite{Fox}. In the other direction, the best-known
construction is due to Bohman \cite{Bohman} (see also \cite{Lunnon}), who showed that there exist arbitrarily large such sets with $a_n \leq 0.22002 \cdot 2^n$.

Several variations on the problem have appeared during the years, such as \cite{MYR} and \cite{CDD}.

Following the notation of \cite{CDD}, we denote by $\mathcal{F}_{\lambda,n}$ the family of all subsets of $[1,n]$ having size at most $\lambda n$, where $\lambda\in (0,1]$ is a given constant. The following problem was introduced in \cite{CDD}.

\begin{Problem}\label{Lambda}
For every positive integer $n$, find the least positive $M=M(n)$ such that there exists a sequence $\Sigma=(a_1,\ldots,a_n)$ of elements of $\mathbb{Z}^k$ with $a_i\in [0,M]^k$ for every $i$ (i.e.\ with every $a_i$ having non-negative coordinates) such that for all distinct $A_1,A_2\in \mathcal{F}_{\lambda,n}$ we have that:
$$\sum_{i\in A_1} a_i\not= \sum_{i\in A_2} a_i.$$
\end{Problem}

In this paper, we propose two progressive variants of the Erd\H{o}s distinct sum problem. The first one, inspired by the recent works \cite{CGS} and in \cite{CS} on Davenport and the Erd\H{o}s-Ginzburg-Ziv constants, is to require that the non-zero evaluations of the $m$-th degree symmetric polynomial are all distinct over the sub-sequences of $\Sigma$.

Given a sequence of integers $\Sigma=(a_1,\ldots, a_n)$ and a subset $A\subseteq [1,n]$, we define the $m$-th (degree) evaluation
$$e^m_{\Sigma}(A)=\sum_{\substack{\{i_1,\ldots, i_m\}\subseteq A\\ i_1<\ldots< i_m}} a_{i_1}\cdots a_{i_m},$$
where we adopt the convention that $e_{\Sigma}^m(A)=0$ if $|A|<m$.

\begin{Problem}\footnote{This problem has been presented at the international conference ``Eurocomb 2023'', see \cite{CDF}.}\label{SmallSize}
For every positive integer $n$, find the least positive $M=M(n)$ such that there exists a sequence $\Sigma=(a_1,\ldots,a_n)$ of integers with $a_i\in [0,M]$ for every $i$ such that for all distinct $A_1,A_2\subseteq [1,n]$ of size at least $m$ we have that:
$$e^m_{\Sigma}(A_1)\not= e^m_{\Sigma}(A_2).$$
\end{Problem}
A sequence as in Problem \ref{SmallSize} will be called \emph{$M$-bounded $m$-th evaluation distinct}. The same terminology applies to sequences in $\mathbb{Z}^k$ such that every entry belongs to $[0,M]^k$ and that the evaluation vectors are all distinct.

By computing the variance of the random variable that associates to a randomly chosen set $A\subseteq [1,n]$ the evaluation $e^m_{\Sigma}(A)$, one finds, surprisingly (these evaluations can not be seen as the values assumed by the sum of independent random variables), that the terms of higher degree disappear, and hence the variance method provides a nontrivial bound on $M$. In order to understand the cause of this behavior, we then investigate a generalization of Problem \ref{SmallSize} that is in the same spirit as Problem \ref{Lambda}. We consider here a sequence $\Sigma$ in $\mathbb{Z}^k$ and we require that the vector-valued evaluations of $e^m_{\Sigma}$ are all distinct.

\begin{Problem}\label{SmallSizeLambda}
For every positive integer $n$, find the least positive $M=M(n)$ such that there exists an $M$-bounded sequence $\Sigma=(a_1,\ldots,a_n)$ in $\mathbb{Z}^k$ such that, for all distinct $A_1,A_2\in \mathcal{F}_{\lambda,n}$ of size at least $m$,
$$e^m_{\Sigma}(A_1) \not= e^m_{\Sigma}(A_2).$$
\end{Problem}

For this very general problem, a symmetry argument allows us to show that the terms of higher degree disappear, explaining therefore the behavior of Problems \ref{Lambda} and \ref{SmallSize}. Our result on Problem \ref{SmallSizeLambda} also improves, in the case $m=1$ and $\lambda<1/2$, Theorem 2.5 of \cite{CDD}, that left this case open.

The paper is organized as follows.
Section $2$ is devoted to providing a direct lower bound on the values of $M$ in Problem \ref{SmallSize}. Using the variance method, we provide a nontrivial bound on $M$; in particular we prove that:
$$M>C_m\cdot 2^{\frac{n}{m}}/n^{1-\frac{1}{2m}}.$$
In Section $3$ we obtain, combining a symmetry argument with the variance method, a nontrivial bound on $M$ for Problem \ref{SmallSizeLambda}. In particular, we prove that:
$$ M\geq  C_{k,m,\lambda} \frac{(1+o(1))|\mathcal{F}_{\lambda,n}|^{\frac{1}{mk}}}{n^{1 - \frac{1}{2m}}}
$$
where $C_{k,m,\lambda}$ is an explicit constant.

Finally, in Section $4$, we derive a probabilistic upper bound on the constant $M$ of Problem \ref{SmallSizeLambda} and we present a direct construction to improve this bound for Problem \ref{SmallSize}.
\section{Lower Bounds via Variance Method: Problem \ref{SmallSize}}
In Problem \ref{SmallSize}, a first lower bound to the value of $M$ can be provided using the pigeonhole principle. Indeed, since the number of non-zero evaluations of $e^m_{\Sigma}$ is $2^n-\sum_{i=0}^{m-1} {n \choose i}=(1+o(1))2^n$, these evaluations are spaced at least by one, and each of these is smaller than $e^m_{\Sigma}([1,n])\leq {n\choose m} M^m\leq n^m/(m!) M^m $, it follows that:
$$M>D_m\cdot 2^{\frac{n}{m}}/n.$$

In this section we show, using the variance method (see \cite{Alon2}, \cite{Erdos} or \cite{Guy}), that it is possible to improve this lower bound. Note that the lower bound we will find in this proof holds also when considering sequences of real numbers. In this case, we consider a sequence $\Sigma$ to be $m$-th evaluation distinct whenever $|e_{\Sigma}^m(A_1)-e_{\Sigma}^m(A_2)|\geq 1$ for any distinct $A_1,A_2\subseteq [1,n]$ of size at least $m$.
\begin{Theorem}\label{teo:multidimguy}
Let $\Sigma=(a_1,\ldots,a_n)$ be an $m$-th evaluation distinct sequence in $\mathbb{Z}$ (resp. $\mathbb{R}$) that is $M$-bounded. Then
$$M>\frac{2^{1-\frac{1}{m}}((m-1)!)^{\frac{1}{m}}}{3^{\frac{1}{2m}}} \frac{ 2^{\frac{n}{m}}}{n^{1-\frac{1}{2m}}} (1+o(1)).$$
\end{Theorem}
\proof
Let $\Sigma=(a_1,\ldots,a_n)$ be an $m$-th evaluation distinct sequence of integers (resp. real numbers). Pick a subset $A$ uniformly at random from $2^{[1,n]}$ and define the real random variable $X = e_\Sigma^m(A)$.
We denote by $\mu \coloneqq \mathbb{E}[X]$ and $\sigma^2 \coloneqq \mathbb{E}[X^2] - \mu^2$ respectively the expected value and the variance of the random variable $X$.

We have that:
$$\mu=\sum_{A\subseteq [1,n]:\ |A|\geq m} \frac{e^m_{\Sigma}(A)}{2^n}.$$
The term $a_{i_1}\dots a_{i_m}$ appears in the evaluation $e^m_{\Sigma}(A)$ precisely when $A$ contains $\{i_1,\ldots,$ $i_m\}$; this happens for $2^{n-m}$ subsets of $[1,n]$.
Therefore, we have that:
$$\mu= \frac{2^{n-m}}{2^n}\sum_{\substack{i_1< i_2< \ldots<i_m\\ i_1,\ldots, i_m\in [1,n]}} a_{i_1}\dots a_{i_m}=\frac{e_{\Sigma}^m([1,n])}{2^m}.$$
By definition of variance, we have that:
$$
2^n\sigma^2=\sum_{A\subseteq [1,n]} (e^m_{\Sigma}(A)-\mu)^2=\sum_{A\subseteq [1,n]} \left( \sum_{\substack{i_1< i_2< \ldots<i_m\\ i_1,\ldots, i_m\in A}} a_{i_1}\dots a_{i_m}- \sum_{\substack{i_1< i_2< \ldots<i_m\\ i_1,\ldots, i_m\in [1,n]}} \frac{a_{i_1}\dots a_{i_m}}{2^m}\right)^2.$$

Thanks to the symmetry of $e^m_{\Sigma}$, there exist coefficients $C_0,\ldots,C_m$ such that the latter sum can be written as follows:

\begin{equation}\label{coefficients}C_0\sum_{\substack{i_1< i_2< \ldots<i_{2m}\\ i_1,\ldots, i_{2m}\in [1,n]}} a_{i_1}\ldots a_{i_{2m}}+C_1\sum_{\substack{i_1< i_2< \ldots<i_{2m-1}\\ i_1,\ldots, i_{2m-1}\in [1,n]}} \sum_{\ell\in [1,2m-1]}a_{i_1}a_{i_2}\ldots a_{i_\ell}^2\dots a_{i_{2m-1}}+ \end{equation}
\begin{equation*}+\ldots+C_{m}\sum_{\substack{i_1< i_2< \ldots<i_{m}\\ i_1,\ldots, i_{m}\in [1,n]}}a_{i_1}^2\dots a_{i_m}^2.\end{equation*}

We now claim that:

\begin{itemize}
\item[(\emph{a})] $C_0=0$;
\item[(\emph{b})] $C_1=2^{n-2m}{2m-2\choose m-1}$;
\item[(\emph{c})] $C_k=O(2^n)$ for every $k\in \{2,\ldots,m\}$.
\end{itemize}

\noindent $(a)$ The coefficient of $a_{i_1}\dots a_{i_{2m}}$ is $ {2m\choose m}$ times that obtained by taking the term $a_{i_1}\dots a_{i_m}$ from the first $(e^m_{\Sigma}(A)-\mu)$ in the product and $a_{i_{m+1}}\dots a_{i_{2m}}$ from the second one.
Here we note that we can choose the monomial $1\cdot a_{i_1}\dots a_{i_m}$ in the first $(e^m_{\Sigma}(A)-\mu)$ and $1\cdot a_{i_{m+1}}\dots a_{i_{2m}}$ in the second one from any set $A$ that contains $\{i_1,\ldots,i_{2m}\}$, i.e.\ in $2^{n-2m}$ possible ways. Similarly, we can choose the monomial $2^{-m}\cdot a_{i_1}\dots a_{i_m}$ in the first $(e^m_{\Sigma}(A)-\mu)$ and $2^{-m}\cdot a_{i_{m+1}}\dots a_{i_{2m}}$ in the second one from any set, i.e.\ in $2^n$ possible ways. Finally, we can choose the monomial $1\cdot a_{i_1}\dots a_{i_m}$ in the first $(e^m_{\Sigma}(A)-\mu)$ and $2^{-m}\cdot a_{i_{m+1}}\dots a_{i_{2m}}$ in $(e^m_{\Sigma}(A)-\mu)$ in the second one from any set that contains $\{i_{m+1},\ldots,i_{2m}\}$, i.e. in $2^{n-m}$ possible ways.
Thus, we have that
$$\frac{C_0}{{2m\choose m}}=2^{n-2m}-2\cdot 2^{n-m}\cdot 2^{-m}+2^n\cdot(2^{-m})^2=0.$$

\noindent $(b)$ The coefficient of $a_{i_1}^2\dots a_{i_2}\dots a_{i_{2m-1}}$ is $ {2m-2\choose m-1}$ times that obtained taking the term $a_{i_1}\dots a_{i_m}$ from the first $(e^m_{\Sigma}(A)-\mu)$ in the product and $a_{i_1} a_{i_{m+1}}\dots a_{i_{2m-1}}$ from the second one. Symmetrically, the same is true for every term $a_{i_1}\dots a_{i_\ell}^2\dots a_{i_{2m-1}}$. Hence, arguing as in $(a)$, we get:
$$\frac{C_1}{{2m-2\choose m-1}}=2^{n-2m+1}-2\cdot 2^{n-m}\cdot 2^{-m}+2^n\cdot(2^{-m})^2=2^{n-2m}.$$

\noindent $(c)$ The coefficient of $a_{i_1}^2\dots a_{i_k}^2 a_{i_{k+1}}\dots a_{i_{2m-k}}$ is $ {2m-2k\choose m-k}$ times that obtained taking the term $a_{i_1}\dots a_{i_m}$ from the first $(e^m_{\Sigma}(A)-\mu)$ in the product and $a_{i_1}\dots a_{i_k} a_{i_{m+1}}\dots a_{i_{2m-k}}$ from the second one. Again by symmetry and reasoning as in $(a)$, we get:
$$\frac{C_k}{{2m-2k\choose m-k}}=2^{n-2m+k}-2\cdot 2^{n-m}\cdot 2^{-m}+2^n\cdot(2^{-m})^2=O(2^{n}).$$

Summing up, we can rewrite \eqref{coefficients} as
\begin{equation}\label{coefficient2} 2^n\sigma^2=C_1\sum_{\substack{i_1< i_2< \ldots<i_{2m-1}\\ i_1,\ldots, i_{2m-1}\in [1,n]}} \sum_{\ell\in [1,2m-1]}a_{i_1}a_{i_2}\ldots a_{i_\ell}^2\dots a_{i_{2m-1}}+\end{equation}
\begin{equation*}+O(2^{n})\left(\sum_{k=2}^m \sum_{\substack{i_1< i_2< \ldots<i_{2m-k}\\ i_1,\ldots, i_{2m-k}\in [1,n]}}\sum_{\substack{\ell_1<\ldots<\ell_k\\ \ell_1,\ldots,\ell_k\in [1,2m-k]}} a_{i_1}a_{i_2}\dots a_{i_{\ell_1}}^2\dots a_{i_{\ell_k}}^2\dots a_{i_{2m-k}}\right).\end{equation*}

In \eqref{coefficient2}, for each value of $k$, we have a sum of ${n\choose 2m-k}\cdot {2m-k\choose k}<\frac{n^{2m-k}}{(2m-2k)!k!}$ terms. Since $M$ is larger than any element of the sequence, we get:
\begin{align}\label{coefficient3A}
2^n\sigma^2&<\frac{n^{2m-1}}{(2m-2)!}{2m-2\choose m-1}2^{n-2m} M^{2m}+O(2^{n}n^{2m-2})M^{2m} \\ &=(1+o(1))\cdot \left(\frac{n^{2m-1}}{((m-1)!)^2}2^{n-2m}M^{2m}\right). \nonumber\end{align}

On the other hand, for $|A|\geq m$, the evaluations $e^m_{\Sigma}(A)$ are all different and spaced at least by one, and hence we have that $(e^m_{\Sigma}(A)-\mu)^2$ assumes at least $\frac{1}{2}(2^n-\sum_{i=0}^{m-1} {n\choose i})$ different values. Since the sum $\sum_{A\subseteq [1,n]} (e^m_{\Sigma}(A)-\mu)^2$ is minimized when the values are around $\mu$ and are spaced by one, we obtain the lower bound:
\begin{equation}\label{lower1}
\frac{1+o(1)}{12}2^{3n}=2\sum_{i=0}^{\frac{1}{2}(2^n-\sum_{i=0}^{m-1} {n\choose i})} i^2\leq 2^n\sigma^2.
\end{equation}
To conclude the proof, it is enough to compare \eqref{coefficient3A} and \eqref{lower1}.
\endproof

\section{Lower Bounds via Variance Method: A Symmetry Argument}

Also in Problem \ref{SmallSizeLambda} one can provide a first lower bound on $M$ using the pigeonhole principle. Indeed, the number of non-zero evaluations of $e^m_{\Sigma}$ is $|\mathcal{F}_{\lambda,n}|-\sum_{i=0}^{m-1} {n \choose i}=(1+o(1))|\mathcal{F}_{\lambda,n}|$, these evaluations are spaced at least by one, and each of these is smaller in each coordinate than ${\lambda n\choose m} M^m\leq \frac{(\lambda n)^m}{m!} M^m$. It follows that:
$$ M>D_{k,m,\lambda}\frac{(1+o(1))|\mathcal{F}_{\lambda,n}|^{1/(km)}}{n}.$$
In this section we show how, using the variance method, one can improve the denominator $n$ in this lower bound. First of all, we consider the $1$-dimensional sequence $(1,1,\ldots,1)$ and we upper-bound its variance. Then with a symmetry argument (see \eqref{coefficients4a} below), we reduce the general case to this one.

In the next three lemmas, we let $\Sigma$ be the sequence of $n$ integers $(1,\ldots,1)$. For a subset $A$ chosen uniformly at random from $\mathcal{F}_{\lambda,n}$, we define the real random variable $\bar{X} = e_\Sigma^m(A)$. Notice that if $|A|=i$, then $e_{\Sigma}^m(A)={i \choose m}$. Finally, we let $\bar{\mu}\coloneqq\mathbb{E}[\bar{X}]$ and $\bar{\sigma}^2 \coloneqq \mathbb{E}[\bar{X}^2] - \bar{\mu}^2$. Notice that:

\begin{equation}\label{expected1}
	|F_{\lambda,n}|\bar{\mu}=\sum_{i=0}^{\lambda n}{n\choose i}{i \choose m}.
\end{equation}
We will also use the fact that, if ${z \choose m}$ is seen as a polynomial in the variable $z$, then:

\begin{equation}\label{binom_as_poly}
	{z\choose m}=\frac{z^m-(\sum_{i=1}^{m-1} i)z^{m-1}+o(z^{m-1})}{m!}.
\end{equation}

\begin{Lemma}\label{VarianceSigma111a}
If $n$ is large enough and $\lambda<1/2$ we have that:
$$\bar{\sigma}^2\leq \frac{\sqrt{n}(\lambda n)^{2m-2}(1+o(1))}{((m-1)!)^2}$$
\end{Lemma}
\proof
Using \eqref{expected1} and the definition of variance we get:
\begin{align}
\nonumber
|\mathcal{F}_{\lambda,n}|\bar{\sigma}^2 & =\sum_{A\in \mathcal{F}_{\lambda,n}}\left( e_\Sigma^m(A) -\bar{\mu}\right)^2 =\sum_{i=0}^{\lambda n}{n\choose i}\left( {i \choose m}-{\lambda n\choose m} -\left(\bar{\mu}-{\lambda n\choose m}\right)\right)^2 \\ \nonumber
&=\left(\sum_{i=0}^{\lambda n}{n\choose i}\left( {i \choose m}-{\lambda n\choose m}\right)^2\right)-\left(\bar{\mu}-{\lambda n\choose m}\right)^2|\mathcal{F}_{\lambda,n}| \\ \label{VarianceLambda1} & <\sum_{i=0}^{\lambda n}{n\choose i}\left( {i \choose m}-{\lambda n\choose m}\right)^2.
\end{align}
Now using \eqref{binom_as_poly}, together with the fact that $i<\lambda n$, we get that if $|i-\lambda n|<\sqrt[4]{n}$ then:
$$\left|{i\choose m}-{\lambda n\choose m}\right|=\frac{\sqrt[4]{n}m(\lambda n)^{m-1}(1+o(1))}{m!}=\frac{\sqrt[4]{n}(\lambda n)^{m-1}(1+o(1))}{(m-1)!},$$
and hence \eqref{VarianceLambda1} becomes
\begin{align}
\nonumber
|\mathcal{F}_{\lambda,n}|\bar{\sigma}^2 &< \sum_{i=0}^{\lambda n-\sqrt[4]{n}}{n\choose i}\left( {i \choose m}-{\lambda n\choose m}\right)^2+\sum_{i=\lambda n-\sqrt[4]{n}+1}^{\lambda n}{n\choose i}\left( {i \choose m}-{\lambda n\choose m}\right)^2 \\ \label{VarianceLambda2} &<\sum_{i=0}^{\lambda n-\sqrt[4]{n}}{n\choose i}\left( {i \choose m}-{\lambda n\choose m}\right)^2+\frac{\sqrt{n}(\lambda n)^{2m-2}(1+o(1))}{((m-1)!)^2}|\mathcal{F}_{\lambda,n}|.
\end{align}
Next, we notice that:
\begin{multline*} \sum_{i=0}^{\lambda n-\sqrt[4]{n}}{n\choose i}\left( {i \choose m}-{\lambda n\choose m}\right)^2<n^{2m}\sum_{i=\sqrt[4]{n}}^{\lambda n}{n\choose i-\sqrt[4]{n}} \\ =n^{2m} \sum_{i=\sqrt[4]{n}}^{\lambda n} {n \choose i}\frac{i (i-1)\cdots (i-\sqrt[4]{n}+2)(i-\sqrt[4]{n}+1)}{(n-i+1)(n-i+2)\cdots (n-(i-\sqrt[4]{n}))}.
 \end{multline*}
Since for $i,j\leq \lambda n$ and $\lambda<1/2$ we have that $\frac{i-j}{(n+1)-(i-j)}\leq \frac{\lambda n}{n+1-\lambda n}<\frac{\lambda }{1-\lambda }\eqqcolon\gamma<1$, it follows from the above computation that:
$$ \sum_{i=0}^{\lambda n-\sqrt[4]{n}}{n\choose i}\left( {i \choose m}-{\lambda n\choose m}\right)^2< n^{2m} |\mathcal{F}_{\lambda,n}|(\gamma)^{\sqrt[4]{n}}<n^{2m-2}|\mathcal{F}_{\lambda,n}|,$$
where in the second inequality we used the fact that for $n$ large enough $(\gamma)^{\sqrt[4]{n}}<\frac{1}{n^2}$. Combining this with \eqref{VarianceLambda2} we get that:
$$\bar{\sigma}^2< \frac{\sqrt{n}(\lambda n)^{2m-2}(1+o(1))}{((m-1)!)^2}.$$
\endproof

\begin{Lemma}\label{VarianceSigma111b}
If $n$ is large enough and $\lambda> 1/2$ we have that:
$$\bar{\sigma}^2\leq
\frac{n^{2m-1}(1+o(1))}{(m!)^2}$$
\end{Lemma}
\proof
Analogously to the proof of Lemma \ref{VarianceSigma111a} we have:
\begin{align} \nonumber |\mathcal{F}_{\lambda,n}|\bar{\sigma}^2 &= \sum_{A\in \mathcal{F}_{\lambda,n}}\left( e_\Sigma^m(A) -\bar{\mu}\right)^2 = \sum_{i=0}^{\lambda n}{n\choose i}\left( {i \choose m}-{n/2\choose m} -\left(\bar{\mu}-{n/2\choose m}\right)\right)^2 \\ \nonumber &=\left(\sum_{i=0}^{\lambda n}{n\choose i}\left( {i \choose m}-{n/2\choose m}\right)^2\right)-\left(\bar{\mu}-{n/2\choose m}\right)^2|\mathcal{F}_{\lambda,n}| \\ \label{VarianceBigLambda1} &<\sum_{i=0}^{\lambda n}{n\choose i}\left( {i \choose m}-{n/2\choose m}\right)^2.
\end{align}
Now using \eqref{binom_as_poly}, together with the fact that $i\le n$, we get:
\begin{align*} \left|{i \choose m}-{n/2\choose m}\right| &=\left|i-n/2\right|\left(\frac{(n/2)^{m-1}+(n/2)^{m-2}i+\ldots+i^{m-1} +o(n^{m-1})}{m!}\right) \\ &<\left|i-n/2\right|\left(\frac{n^{m-1}((1/2)^{m-1}+(1/2)^{m-2}+\ldots+1 +o(1))}{m!}\right) \\
&<\left|i-n/2\right|\left(\frac{2n^{m-1}(1+o(1))}{m!}\right), 
\end{align*}
and hence by \eqref{VarianceBigLambda1} we deduce:
\begin{equation}\label{VarianceBigLambda2}|\mathcal{F}_{\lambda,n}|\bar{\sigma}^2<\left(\sum_{i=0}^{\lambda n}{n\choose i}\left(i-n/2\right)^2\right)\left(\frac{2n^{m-1}(1+o(1))}{m!}\right)^2.\end{equation}
Now notice that:
$$\sum_{i=0}^{\lambda n}{n\choose i}\left(i-n/2\right)^2<\sum_{i=0}^{n}{n\choose i}\left(i-n/2\right)^2,$$
and hence setting $A\coloneqq\sum_{i=0}^{n}{n\choose i}i$ and $B\coloneqq \sum_{i=0}^{ n}{n\choose i}i^2$ we have:
$$\sum_{i=0}^{n}{n\choose i}\left(i-n/2\right)^2=-nA+B+\frac{n^2}{4}\sum_{i=0}^{n}{n\choose i}=-nA+B+n^22^{n-2}.$$
Since
$$A=\sum_{i=0}^{n}{n\choose i}i=n\sum_{i=0}^{ n}{n-1\choose i-1}=n2^{n-1}$$
and
\begin{align*} &\qquad\qquad B=\sum_{i=0}^{n}{n\choose i}i^2= n\sum_{i=0}^{ n}{n-1\choose i-1}i=n\sum_{i=0}^{ n}{n-1\choose i-1}(i-1+1)= \\ &n(n-1)\sum_{i=0}^{ n}{n-2\choose i-2}+n\sum_{i=0}^{ n}{n-1\choose i-1}=n(n-1)2^{n-2}+n2^{n-1}=n(n+1)2^{n-2}, \end{align*}
it follows that
$$-nA+B+n^22^{n-2}=n2^{n-2}.$$
Therefore, we can deduce from \eqref{VarianceBigLambda2} that:
\begin{equation*}|\mathcal{F}_{\lambda,n}|\bar{\sigma}^2<n2^{n-2}\left(\frac{2n^{m-1}(1+o(1))}{m!}\right)^2.\end{equation*}
Since for $\lambda>1/2$ we have $|\mathcal{F}_{\lambda,n}|=2^{n}(1+o(1))$, the claim follows.
\endproof
\begin{Lemma}\label{VarianceSigma111c}
If $n$ is large enough and $\lambda=1/2$ we have that:
$$\bar{\sigma}^2\leq
\frac{n^{2m-1}(1+o(1))}{2^{2m}((m-1)!)^2}$$
\end{Lemma}
\proof
Arguing as at the beginning of the proof of Lemma \ref{VarianceSigma111b}, we get to:
\begin{equation}\label{VarianceIII}
	|\mathcal{F}_{\lambda,n}|\bar{\sigma}^2<\sum_{i=0}^{\lfloor n/2\rfloor}{n\choose i}\left( {i \choose m}-{n/2\choose m}\right)^2.
\end{equation}
Now using \eqref{binom_as_poly} together with the fact that $i \leq n/2$, we get that: 
\begin{align*} \left|{i \choose m}-{n/2\choose m}\right| &=\left|i-n/2\right|\left(\frac{(n/2)^{m-1}+(n/2)^{m-2}i+\ldots+i^{m-1} +o(n^{m-1})}{m!}\right) \\ &<\left|i-n/2\right|\left(\frac{n^{m-1}(1+o(1))}{2^{m-1}(m-1)!}\right).\end{align*}
Hence from \eqref{VarianceIII} we get:
\begin{equation}\label{VarianceBigLambda2b}|\mathcal{F}_{\lambda,n}|\bar{\sigma}^2<\left(\sum_{i=0}^{\lfloor n/2\rfloor}{n\choose i}\left(i-n/2\right)^2\right)\left(\frac{n^{m-1}(1+o(1))}{2^{m-1}(m-1)!}\right)^2.\end{equation}
Next, notice that:
$$\sum_{i=0}^{\lfloor n/2\rfloor}{n\choose i}\left(i-n/2\right)^2=\frac{1}{2}\sum_{i=0}^{n}{n\choose i}\left(i-n/2\right)^2=n2^{n-3}.$$
It follows from \eqref{VarianceBigLambda2b} that:
\begin{equation*}|\mathcal{F}_{1/2,n}|\bar{\sigma}^2<n2^{n-3}\left(\frac{n^{m-1}(1+o(1))}{2^{m-1}(m-1)!}\right)^2.\end{equation*}
Since for $\lambda=1/2$ we have $|\mathcal{F}_{1/2,n}|=2^{n-1}(1+o(1))$, the claim follows.
\endproof
Summing up Lemmas \ref{VarianceSigma111a}, \ref{VarianceSigma111b} and \ref{VarianceSigma111c}, we get the following corollary.
\begin{Corollary}\label{VarianceSigma111}
Let $\Sigma$ be the sequence of $n$ integers $(1,\ldots,1)$. Pick a subset $A$ uniformly at random from $\mathcal{F}_{\lambda,n}$ and define the real random variable $\bar{X} = e_\Sigma^m(A)$. Setting $\bar{\mu}\coloneqq\mathbb{E}[\bar{X}]$ and $\bar{\sigma}^2 \coloneqq \mathbb{E}[\bar{X}^2] - \bar{\mu}^2$, for a large enough $n$ we have that:
$$\bar{\sigma}^2\leq \begin{dcases} \frac{\sqrt{n}(\lambda n)^{2m-2}(1+o(1))}{((m-1)!)^2} & \mbox{if } \lambda<1/2\\
\frac{n^{2m-1}(1+o(1))}{2^{2m}((m-1)!)^2} & \mbox{if } \lambda=1/2\\
\frac{n^{2m-1}(1+o(1))}{(m!)^2} & \mbox{if } \lambda>1/2
\end{dcases}.$$
\end{Corollary}

Next, we are going to use Corollary \ref{VarianceSigma111} to deduce upper bounds on the variance for a generic $k$-dimensional sequence.

\begin{Lemma}\label{teo:multidimguy2a}
Let $\Sigma=(a_1,\ldots,a_n)$ be a sequence in $\mathbb{Z}^k$ that is $M$-bounded and such that, for all distinct $A_1,A_2\in \mathcal{F}_{\lambda,n}$ of size at least $m$,
$$e^m_{\Sigma}(A_1) \not= e^m_{\Sigma}(A_2).$$
Pick a subset $A$ uniformly at random from $\mathcal{F}_{\lambda,n}$ and define the real random variable $\bar{X} = e_\Sigma^m(A)$.
Setting $\mu\coloneqq\mathbb{E}[\bar{X}]$ and $\sigma^2\coloneqq\mathbb{E}[\bar{X}^2] - \bar{\mu}^2$, for a large enough $n$ we have that:
$$\sigma^2\leq\begin{dcases}\frac{2k(\lambda^{2m-1}+\lambda^{2m}) n^{2m-1}(1+o(1))}{((m-1)!)^2} M^{2m} & \mbox{if }\lambda< 1/2 \\
\frac{k(2^{2m+1}\lambda^{2m-1}+2^{2m+1}\lambda^{2m}+1) n^{2m-1}(1+o(1))}{2^{2m}((m-1)!)^2} M^{2m} & \mbox{if }\lambda= 1/2 \\
\frac{k(2m^2\lambda^{2m-1}+2m^2\lambda^{2m}+1) n^{2m-1}(1+o(1))}{(m!)^2} M^{2m} & \mbox{if }\lambda>1/2\end{dcases}.$$

\end{Lemma}
\proof
We start from the case $k=1$. Notice that
\begin{equation}\label{expected}
\mu=\frac{\sum_{A\in \mathcal{F}_{\lambda,n}}\sum_{\substack{i_1< i_2< \ldots<i_m\\ i_1,\ldots, i_m\in A}} a_{i_1}\dots a_{i_m}}{|\mathcal{F}_{\lambda,n}|}.
\end{equation}

By definition of variance, we have that:
$$
|\mathcal{F}_{\lambda,n}|\sigma^2=\sum_{A\in \mathcal{F}_{\lambda,n}} (e^m_{\Sigma}(A)-\mu)^2=\sum_{A\in \mathcal{F}_{\lambda,n}} \left( \sum_{\substack{i_1< i_2< \ldots<i_m\\ i_1,\ldots, i_m\in A}} a_{i_1}\dots a_{i_m}\right)^2- \mu^2|\mathcal{F}_{\lambda,n}|.
$$
Let $C_0,\ldots, C_m$ be such that:

\begin{equation}\label{coefficients1}\sum_{A\in \mathcal{F}_{\lambda,n}} \left( \sum_{\substack{i_1< i_2< \ldots<i_m\\ i_1,\ldots, i_m\in A}} a_{i_1}\dots a_{i_m}\right)^2=C_0\sum_{\substack{i_1< i_2< \ldots<i_{2m}\\ i_1,\ldots, i_{2m}\in [1,n]}} a_{i_1}\ldots a_{i_{2m}}+ \end{equation}
\begin{equation*}+C_1\sum_{\substack{i_1< i_2< \ldots<i_{2m-1}\\ i_1,\ldots, i_{2m-1}\in [1,n]}} \sum_{\ell\in [1,2m-1]}a_{i_1}a_{i_2}\ldots a_{i_\ell}^2\dots a_{i_{2m-1}}+\ldots+C_{m}\sum_{\substack{i_1< i_2< \ldots<i_{m}\\ i_1,\ldots, i_{m}\in [1,n]}}a_{i_1}^2\dots a_{i_m}^2,\end{equation*}

and let $D_0,\ldots, D_m$ be such that:

\begin{equation}\label{coefficients2}\mu^2|\mathcal{F}_{\lambda,n}|=D_0\sum_{\substack{i_1< i_2< \ldots<i_{2m}\\ i_1,\ldots, i_{2m}\in [1,n]}} a_{i_1}\ldots a_{i_{2m}}+ \end{equation}
\begin{equation*}+D_1\sum_{\substack{i_1< i_2< \ldots<i_{2m-1}\\ i_1,\ldots, i_{2m-1}\in [1,n]}} \sum_{\ell\in [1,2m-1]}a_{i_1}a_{i_2}\ldots a_{i_\ell}^2\dots a_{i_{2m-1}}+\ldots+D_{m}\sum_{\substack{i_1< i_2< \ldots<i_{m}\\ i_1,\ldots, i_{m}\in [1,n]}}a_{i_1}^2\dots a_{i_m}^2.\end{equation*}

We now claim that:

\begin{itemize}
\item[(\emph{a})] $|C_1|\leq {2m-2\choose m-1} \lambda^{2m-1}|\mathcal{F}_{\lambda,n}|$;
\item[(\emph{b})] $|D_1|\leq {2m-2\choose m-1} \lambda^{2m}|\mathcal{F}_{\lambda,n}|$;
\item[(\emph{c})] $C_j=O(|\mathcal{F}_{\lambda,n}|)=D_j$ for every $j\in \{2,\ldots,m\}$.
\end{itemize}

\noindent $(a)$ The coefficient of $a_{i_1}^2\dots a_{i_2}\dots a_{i_{2m-1}}$ in \eqref{coefficients1} is $ {2m-2\choose m-1}$ times that obtained taking the term $a_{i_1}\dots a_{i_m}$ from the first $e^m_{\Sigma}(A)$ in the product and $a_{i_1} a_{i_{m+1}}\dots a_{i_{2m-1}}$ from the second one. Symmetrically, the same is true for every term $a_{i_1}\dots a_{i_\ell}^2\dots a_{i_{2m-1}}$. Since a set $A$ containing $i_1,\ldots,i_{2m-1}$ contains $|A|-(2m-1)$ further elements, we get:
$$\frac{C_1}{{2m-2\choose m-1}}=\sum_{i=2m-1}^{\lambda n} {n-(2m-1)\choose i-(2m-1)}=\sum_{i=2m-1}^{\lambda n} \frac{(n-(2m-1))!}{(i-(2m-1))!(n-i)!}.$$
Since $i\leq \lambda n$ and $\lambda\le 1$, we have that:
$$\frac{(n-(2m-1))!}{(i-(2m-1))!(n-i)!}= \frac{n!}{i!(n-i)!}\frac{i(i-1)\cdots (i-2m+2)}{n(n-1)\cdots (n-2m+2)}\leq {n\choose i}\lambda^{2m-1}.$$
It follows that:
$$\frac{C_1}{{2m-2\choose m-1}}\leq \sum_{i=2m-1}^{\lambda n} {n\choose i}\lambda^{2m-1}\leq |\mathcal{F}_{\lambda,n}|\lambda^{2m-1}.$$

\noindent $(b)$ Arguing as in point $(a)$ and using \eqref{expected}, we see that the coefficient of $a_{i_1}^2\dots a_{i_2}\dots a_{i_{2m-1}}$ in $\mu^2$ is $ {2m-2\choose m-1}$ times that obtained taking the term $a_{i_1}\dots a_{i_m}$ from the first $\mu$ in the product and $a_{i_1} a_{i_{m+1}}\dots a_{i_{2m-1}}$ from the second one. Symmetrically, the same is true for every term $a_{i_1}\dots a_{i_\ell}^2\dots a_{i_{2m-1}}$. Since a set $A$ containing $i_1,\ldots,i_m$ contains $|A|-m$ further elements, the term $a_{i_1}\dots a_{i_m}$ appears $\sum_{i=m}^{\lambda n} {n-m\choose i-m} $ times in the sum. Since $\frac{\sum_{i=m}^{\lambda n} {n-m\choose i-m} }{|\mathcal{F}_{\lambda,n}|}\leq \lambda^m$, proceeding as in point $(a)$ we get:
$$\frac{D_1}{|\mathcal{F}_{\lambda,n}|{2m-2\choose m-1}}\leq \lambda^{2m}.$$

\noindent $(c)$ The coefficient of $a_{i_1}^2\dots a_{i_j}^2 a_{i_{j+1}}\dots a_{i_{2m-j}}$ in \eqref{coefficients1} is $ {2m-2j\choose m-j}$ times that obtained taking the term $a_{i_1}\dots a_{i_m}$ from the first $e^m_{\Sigma}(A)$ in the product and $a_{i_1}\dots a_{i_j}$ $a_{i_{m+1}}\dots a_{i_{2m-j}}$ from the second one. Chosen these $2m-j$ terms, we need to add $|A|-(2m-j)$ terms to complete $A$. Reasoning as in $(a)$, we get:
$$\frac{C_j}{{2m-2j\choose m-j}}=\sum_{i=2m-1}^{\lambda n} {n-(2m-j)\choose i-(2m-j)}=O(|\mathcal{F}_{\lambda,n}|).$$
Arguing analogously, one sees that the coefficient of $a_{i_1}^2\dots a_{i_j}^2 a_{i_{j+1}}\dots a_{i_{2m-j}}$ in \eqref{coefficients2} is
$$\frac{D_j}{{2m-2j\choose m-j}}=O(|\mathcal{F}_{\lambda,n}|).$$

Claim $(c)$ allows us to write:
\begin{equation}\label{coefficient3} |\mathcal{F}_{\lambda,n}|\sigma^2=(C_0-D_0)\sum_{\substack{i_1< i_2< \ldots<i_{2m}\\ i_1,\ldots, i_{2m}\in [1,n]}} a_{i_1}\ldots a_{i_{2m}}+\end{equation}
\begin{equation*}+(C_1-D_1)\sum_{\substack{i_1< i_2< \ldots<i_{2m-1}\\ i_1,\ldots, i_{2m-1}\in [1,n]}} \sum_{\ell\in [1,2m-1]}a_{i_1}a_{i_2}\ldots a_{i_\ell}^2\dots a_{i_{2m-1}}+\end{equation*}
\begin{equation*}+O(|\mathcal{F}_{\lambda,n}|)\left(\sum_{j=2}^m \sum_{\substack{i_1< i_2< \ldots<i_{2m-j}\\ i_1,\dots, i_{2m-j}\in [1,n]}}\sum_{\substack{\ell_1<\ldots<\ell_j\\ \ell_1,\ldots,\ell_j\in [1,2m-j]}} a_{i_1}a_{i_2}\dots a_{i_{\ell_1}}^2\dots a_{i_{\ell_j}}^2\dots a_{i_{2m-j}}\right).\end{equation*}

Now we note that if each $a_i$ is $1$, then Corollary \ref{VarianceSigma111} provides us an upper bound on $\sigma^2$. Hence, treating \eqref{coefficient3} as a polynomial expression in the $a_i$'s and evaluating it at $(a_1,\ldots,a_n)=(1,\ldots,1)$, one obtains:
\begin{equation*} (C_0-D_0){n \choose 2m}+(C_1-D_1){n \choose 2m-1}(2m-1)+O(|\mathcal{F}_{\lambda,n}|n^{2m-2})\le L,\end{equation*}
	where $L$ is the upper bound provided by Corollary \ref{VarianceSigma111}. From now on, we assume $\lambda<1/2$. The remaining two cases are treated in an analogous way.
	
	For $\lambda<1/2$ we have that $L=\frac{\sqrt{n}(\lambda n)^{2m-2}(1+o(1))}{((m-1)!)^2}|\mathcal{F}_{\lambda,n}|$, and hence
	
\begin{equation}\label{coefficients4a} (C_0-D_0){n \choose 2m}+(C_1-D_1)(2m-1){n \choose 2m-1}+O(|\mathcal{F}_{\lambda,n}|\sqrt{n}n^{2m-2})=0.\end{equation}
Now let us go back to \eqref{coefficient3}, and assume that $(a_1,\ldots,a_n)$ is $M$-bounded. Then we have:

\begin{equation*} |\mathcal{F}_{\lambda,n}|\sigma^2\leq |C_0-D_0|\sum_{\substack{i_1< i_2< \ldots<i_{2m}\\ i_1,\ldots, i_{2m}\in [1,n]}} a_{i_1}\ldots a_{i_{2m}}+\end{equation*}
\begin{equation*}+|C_1-D_1|\sum_{\substack{i_1< i_2< \ldots<i_{2m-1}\\ i_1,\ldots, i_{2m-1}\in [1,n]}} \sum_{\ell\in [1,2m-1]}a_{i_1}a_{i_2}\ldots a_{i_\ell}^2\dots a_{i_{2m-1}}+\end{equation*}
\begin{equation*}+O(|\mathcal{F}_{\lambda,n}|)\left(\sum_{j=2}^m \sum_{\substack{i_1< i_2< \ldots<i_{2m-j}\\ i_1,\ldots, i_{2m-j}\in [1,n]}}\sum_{\substack{\ell_1<\ldots<\ell_j\\ \ell_1,\ldots,\ell_j\in [1,2m-j]}} a_{i_1}a_{i_2}\dots a_{i_{\ell_1}}^2\dots a_{i_{\ell_j}}^2\dots a_{i_{2m-j}}\right)\leq \end{equation*}
\begin{equation*} |C_0-D_0| {n \choose 2m}M^{2m}+|C_1-D_1|(2m-1){n \choose 2m-1}M^{2m}+O(|\mathcal{F}_{\lambda,n}|n^{2m-2})M^{2m}.\end{equation*}
Combining this with \eqref{coefficients4a} we get:
\begin{equation*}|\mathcal{F}_{\lambda,n}|\sigma^2\leq 2|C_1-D_1|(2m-1){n \choose 2m-1}M^{2m}+O(|\mathcal{F}_{\lambda,n}|\sqrt{n}n^{2m-2})M^{2m},\end{equation*}
and using claims $(a)$ and $(b)$, we deduce that:

\begin{equation*}\label{coefficients6} |\mathcal{F}_{\lambda,n}|\sigma^2\leq 2(\lambda^{2m-1}+\lambda^{2m}){2m-2\choose m-1}(2m-1) {n \choose 2m-1}(1+o(1))|\mathcal{F}_{\lambda,n}|M^{2m}\end{equation*}
\begin{equation*}\leq \frac{2(\lambda^{2m-1}+\lambda^{2m}) n^{2m-1}(1+o(1))}{((m-1)!)^2} |\mathcal{F}_{\lambda,n}|M^{2m}.\end{equation*}

Finally, if the the entries of $\Sigma$ belong to $\mathbb{Z}^k$ for $k\ge 2$, one can just argue componentwise, obtaining
$$ |\mathcal{F}_{\lambda,n}|\sigma^2\leq \frac{2k(\lambda^{2m-1}+\lambda^{2m}) n^{2m-1}(1+o(1))}{((m-1)!)^2} |\mathcal{F}_{\lambda,n}|M^{2m},$$
namely the claim.
\endproof

\begin{Theorem}\label{teo:multidimguy2}
Let $\Sigma=(a_1,\ldots,a_n)$ be an $M$-bounded sequence in $\mathbb{Z}^k$ such that, for all distinct $A_1,A_2\in \mathcal{F}_{\lambda,n}$ of size at least $m$,
$$e^m_{\Sigma}(A_1) \not= e^m_{\Sigma}(A_2).$$
Then
$$ M^{2m}\geq \begin{dcases}
\left(\frac{((m-1)!)^2\Gamma(k/2 + 1)^{2/k}}{2(\lambda^{2m-1}+\lambda^{2m})\pi(k+2)}\right) \frac{(1+o(1))|\mathcal{F}_{\lambda,n}|^{2/k}}{n^{2m-1}} & \mbox{ if } \lambda< 1/2\\
\left(\frac{2^{2m}((m-1)!)^2\Gamma(k/2 + 1)^{2/k}}{(2^{2m+1}\lambda^{2m-1}+2^{2m+1}\lambda^{2m}+1)\pi(k+2)}\right) \frac{(1+o(1))|\mathcal{F}_{\lambda,n}|^{2/k}}{n^{2m-1}} & \mbox{ if } \lambda= 1/2\\
\left(\frac{(m!)^2\Gamma(k/2 + 1)^{2/k}}{(2m^2\lambda^{2m-1}+2m^2\lambda^{2m}+1)\pi(k+2)}\right) \frac{(1+o(1))|\mathcal{F}_{\lambda,n}|^{2/k}}{n^{2m-1}} & \mbox{ if } \lambda> 1/2. \end{dcases}
$$
\end{Theorem}
\proof
To prove the claim, we combine Lemma \ref{teo:multidimguy2a} with a lower bound on $|\mathcal{F}_{\lambda,n}|\sigma^2$. The idea here is that the minimal value of $|\mathcal{F}_{\lambda,n}|\sigma^2=\sum_{A\in \mathcal{F}_{\lambda,n}} |e^m_{\Sigma}(A)-\mu|^2$ is obtained when the $e_{\Sigma}^m(A)$'s are packed as close as possible around $\mu$. In other words, they have to fill the $k$-dimensional ball of volume $|\mathcal{F}_{\lambda,n}|-\sum_{i=0}^{m-1} {n \choose i}=(1+o(1))|\mathcal{F}_{\lambda,n}|$ centered in $\mu$. The radius $R$ of such ball is given by:
$$
R = \frac{\Gamma(k/2 + 1)^{1/k}}{\sqrt{\pi}} (1+o(1)) |\mathcal{F}_{\lambda,n}|^{1/k}.
$$
Then, following the computations of \cite[p. 176]{CDD}, we obtain:
\begin{align}\nonumber
|\mathcal{F}_{\lambda,n}|\sigma^2 &\geq R^{2+k} \frac{k\pi^{k/2}}{\Gamma(k/2 + 1)(k+2)}  (1+o(1))\\ \label{eq:spherevolume}
&= \frac{k \Gamma(k/2 + 1)^{2/k}}{\pi (k+2)} (1+o(1))  |\mathcal{F}_{\lambda,n}|^{2/k+1}.
\end{align}

To conclude, it is enough to compare \eqref{eq:spherevolume} with Lemma \ref{teo:multidimguy2a}.
\endproof

\begin{Remark}
Theorem \ref{teo:multidimguy2} can be generalized to sequences $\Sigma$ in $\mathbb{R}^k$ such that for all distinct $A_1, A_2 \in \mathcal{F}_{\lambda, n}$ we have $|e^m_{\Sigma}(A_1) - e^m_{\Sigma}(A_2)|\geq 1$, at the cost of dividing the lower bounds by a factor of $4$. Here we give a sketch of the proof. The key idea is that $\sum_{A\in \mathcal{F}_{\lambda,n}} |e^m_{\Sigma}(A)-\mu|^2$ is lower bounded by $\sum_{i=1}^{|\mathcal{F}_{\lambda,n}|} |z_i|^2$ where $z_1 = 0$ and $z_i \in \left(\frac{1}{2} \Z\right)^k$ for $i \geq 2$. This can be proved by noticing that each real number inside the ball centered at $0$ of radius $R\coloneqq \max_i |z_i|$ is contained in two balls of radius $1/2$ centered at some $z_i,z_j$. On the other hand, since $|e^m_{\Sigma}(A_i) - e^m_{\Sigma}(A_j)|\geq 1$ two balls of radius $1/2$ centered in $e^m_{\Sigma}(A_i)$ and $e^m_{\Sigma}(A_j)$ cannot intersect unless they coincide, and therefore for every $R'\le R$ the ball centered at $0$ of radius $R'$ contains more $z_i$'s than $e^m_{\Sigma}(A_i) - \mu$'s. The lower bound immediately follows. Now $\sum_{i=1}^{|\mathcal{F}_{\lambda,n}|} |z_i|^2$ is exactly what we compute in~\eqref{eq:spherevolume}, divided by a factor of $4$.
\end{Remark}

\begin{Remark}
Theorem \ref{teo:multidimguy2} provides a lower bound of the following form:
$$ M\geq C_{k,m,\lambda}\frac{ n^{1/(2m)} (1+o(1)) |\mathcal{F}_{\lambda,n}|^{1/(mk)}}{n},
$$
improving by a factor of $n^{1/(2m)}$ the trivial pigeonhole lower bound. Since the bound holds for any $\lambda$, we also improve Theorem 2.5 of \cite{CDD}, where the case $\lambda<1/2$ was left open.

On the other hand, if $\lambda=k=1$, the bound of Theorem \ref{teo:multidimguy2} becomes
$$M\geq C_{1,m,1} \frac{ 2^{\frac{n}{m}}}{n^{1-\frac{1}{2m}}} (1+o(1))$$
which has the same trend but a worse constant than the result of Theorem \ref{teo:multidimguy} which is of the form,
$$M\geq C_m \frac{ 2^{\frac{n}{m}}}{n^{1-\frac{1}{2m}}} (1+o(1))$$
where
$$C_{1,m,1}=\frac{(m!)^{1/m}}{\sqrt[2m]{3(4m^2+1)}}<\frac{2^{1-\frac{1}{m}}((m-1)!)^{\frac{1}{m}}}{3^{\frac{1}{2m}}}=C_m.$$
Indeed Corollary \ref{VarianceSigma111} and \eqref{coefficients4a} explain why the terms of higher degree can be replaced but a careful computation of the coefficients (which we are able to do only in some particular cases) allows us to further improve the bound.
\end{Remark}

\section{Upper bounds}

In this section, we provide an upper bound to the value of $M$ in Problem \ref{SmallSizeLambda} using the probabilistic method. In particular, we show that $M \leq 4^{n h(\lambda)/k + o(n)}$, where 
$$h(\lambda) = -\lambda \log_2 \lambda - (1-\lambda) \log_2 (1-\lambda)$$
denotes the binary entropy function. Then, using a direct construction (that we provide for $\lambda = k = 1$ but can be generalized to $\lambda\leq 1$ and $k\geq 1$), we prove that $M \leq 2^{n + o(n)}$. It can be shown computationally that $h (\lambda) < 1/2$ for every $\lambda \in [0, 0.11]$; this means that when $k=1$ and $\lambda\in [0,0.11]$ the probabilistic method improves the result obtained using our direct construction.

In order to proceed, we first need to recall the following celebrated lemma \cite{Schwartz, Zippel}.

\begin{Lemma}[Schwartz-Zippel Lemma]\label{lem:SZ}
Let $\mathbb F$ be a field and $P \in \mathbb{F}[x_1, x_2, \ldots, x_t]$ be a non-zero polynomial of degree $d$. Consider a finite subset $A \subseteq \mathbb{F}$. If $r_1, r_2, \ldots, r_t$ are picked uniformly at random from $A$, then
$$
\Pr[P(r_1, r_2, \ldots, r_t) = 0] \leq \frac{d}{|A|}\,.
$$
\end{Lemma}

\begin{Theorem}\label{thm: ubprob}
Let $\lambda$ be a real number in $[0, 1/2)$ and let
$$
C_{\lambda, k} = \sqrt[k]{\frac{4^{h(\lambda) \tau_{\lambda}}}{\tau_{\lambda}}} \text{ and } \tau_{\lambda} = \left \lceil \frac{1}{4^{h(\lambda)}-1} \right \rceil.
$$
Then there exists a $k$-dimensional sequence $\Sigma=(a_1,\ldots,a_n)$ of $\left( C_{\lambda,k} \cdot m \cdot 4^{h(\lambda) n/k}\right)$-bounded elements such that, for all distinct $A_1,A_2\in \mathcal{F}_{\lambda,n}$ of size at least $m$,
$$e^m_{\Sigma}(A_1) \not= e^m_{\Sigma}(A_2).$$
\end{Theorem}
\begin{proof}
We choose, uniformly at random, a sequence $\Sigma'$ with entries in $[1, M]^k$ and length $n'$ (where the values of $M$ and $n'$ will be specified later). Let $X$ be the random variable that represents the numbers of pairs of elements of $\mathcal{F}_{\lambda,n'}$ such that $e^m_{\Sigma'}(A_1) = e^m_{\Sigma'}(A_2)$.

Then we need to estimate the following expected value
\begin{align*}
\mathbb{E}[X] & =\mathbb{E}(|\{\{A_1,A_2\}: e^m_{\Sigma'}(A_1)=e^m_{\Sigma'}(A_2), A_1,A_2\in \mathcal{F}_{\lambda,n'}, |A_1| \geq m, |A_2| \geq m\}|)\\
& = \sum_{\{A_1,A_2\}:\ A_1,A_2\in \mathcal{F}_{\lambda,n'}, |A_1|\geq m, |A_2| \geq m} \Pr[e^m_{\Sigma'}(A_1)=e^m_{\Sigma'}(A_2)].
\end{align*}
Since each of the $k$ components of $e^m_{\Sigma'}(A_1) - e^m_{\Sigma'}(A_2)$ is a polynomial of degree $m$ and each element's component in $\Sigma'$ is picked independently and uniformly at random in $[1,M]$, then by Lemma \ref{lem:SZ} we have that
$$\Pr[e^m_{\Sigma'}(A_1) - e^m_{\Sigma'}(A_2) = 0] \leq \left(\frac{m}{M}\right)^k\,.$$
Then we have that:
\begin{align*}
\mathbb{E}[X] & \leq \left(\frac{m}{M}\right)^k |\{\{A_1,A_2\}:\ A_1,A_2\in \mathcal{F}_{\lambda,n'}, |A_1| \geq m, |A_2| \geq m\}|.
\end{align*}
Thanks to the well-known inequality $\sum_{i=0}^\ell \binom{r}{i} \leq 2^{h(\ell/r) r}$ for $\ell < r/2$, we obtain:
\begin{equation} \label{eq:probabilisticboundE}\mathbb{E}[X]< \left(\frac{m}{M}\right)^k \cdot \left(\sum_{i=m}^{\lambda n'} \binom{n'}{i} \right)^2 \leq \left(\frac{m}{M}\right)^k 2^{2 h(\lambda) n'}\,.\end{equation}

This means that if $(m / M)^k 2^{2 h(\lambda) n'} \leq t$, there exists a sequence $\Sigma'=(a_1,\ldots,a_{n'})$ of elements in $\mathbb{Z}^k$ with at most $t$ pairs $\{A_1, A_2\}$ that have the same evaluation and satisfy the assumptions. Hence, we can remove $t$ elements from $\Sigma'$ and obtain a new sequence $\Sigma=(a_1,\ldots,a_n)$, with $n = n' - t$ elements, such that $e^m_{\Sigma}(A_1) \neq e^m_{\Sigma}(A_2)$ for all distinct $A_1, A_2 \in \mathcal{F}_{\lambda, n}$. Thanks to \eqref{eq:probabilisticboundE}, we deduce that a sequence $\Sigma$ as in the claim exists whenever
\begin{equation} \label{eq:probabilisticbound}
M \geq \sqrt[k]{\frac{4^{h(\lambda) t}}{t}} \cdot m \cdot 4^{h(\lambda) n/k}.
\end{equation}
One can check that the function $g_{\lambda}(t) \coloneqq \frac{4^{h(\lambda) t}}{t}$ is strictly convex for $t > 0$ and that the minimum integer $\gamma$ for which $g_{\lambda}(\gamma+1) \geq g_{\lambda}(\gamma)$ is equal to $\tau_{\lambda}$. Therefore $t = \tau_{\lambda} = \left \lceil \frac{1}{4^{h(\lambda)}-1} \right \rceil$ is the best choice to optimize \eqref{eq:probabilisticbound}.
\end{proof}

Using the same method of Theorem \ref{thm: ubprob} for $\lambda = 1$, one can easily obtain the following theorem.

\begin{Theorem}\label{thm: upprobm}
There exists a sequence of $n$ elements of $\mathbb{Z}^k$ that is $m$-th evaluation distinct and $M$-bounded such that
$$
M \leq m \cdot 4 ^{n/k}\,.
$$
\end{Theorem}

Now we provide a direct construction of a sequence $\Sigma$ that is $m$-th evaluation distinct. We start by constructing such sequence on the real numbers and then we adapt the idea over the integers.

\begin{Lemma}\label{lem: upper}
Let $\epsilon$ be a positive real number and let $m \geq 2$ be an integer. For every $n$ large enough the sequence of real numbers $\Sigma = (a_1, a_2, \ldots, a_n)$, where
$$a_i = (2+\epsilon)^n - 2^{i-1}\mbox{ for $i=1,2,\ldots,n$},$$
is $m$-th evaluation distinct.
\end{Lemma}

\begin{proof}
Suppose by contradiction there exists two distinct subsets $B, C \subseteq [1,n]$ with $|B|,|C|\ge m$ such that
\begin{equation}\label{eq: eqcon}
|e_{\Sigma}^m (B) - e_{\Sigma}^m (C)| < 1\,.
\end{equation}
For an arbitrary subset $S \subseteq [1,n]$ with $|S| \geq m$, by definition we have:
\begin{equation}\label{m_evaluation}
e_{\Sigma}^m (S) = \sum_{j=0}^m (-1)^j (2+\epsilon)^{(m-j)n} \binom{|S|-j}{m-j} \sum_{\substack{\{i_1, i_2, \ldots, i_j\} \subseteq S \\ i_1 < i_2 < \ldots < i_j}} 2^{i_1 + i_2 + \ldots + i_j - j}\,.
\end{equation}
We first show that Inequality \eqref{eq: eqcon} implies $|B| = |C|$. Suppose without loss of generality that $|B| > |C|$. Then \eqref{m_evaluation} implies that:
\begin{multline}\label{eq: diffeval}
e_{\Sigma}^m (B) - e_{\Sigma}^m (C) = (2+\epsilon)^{mn} \left[ \binom{|B|}{m} - \binom{|C|}{m} \right] + \sum_{j=1}^m (-1)^j (2+\epsilon)^{(m-j)n} \\ \left[\binom{|B|-j}{m-j} \sum_{\substack{\{i_1, i_2, \ldots, i_j\} \subseteq B \\ i_1 < i_2 < \ldots < i_j}} 2^{i_1 + i_2 + \ldots + i_j - j} - \binom{|C|-j}{m-j} \sum_{\substack{\{i_1, i_2, \ldots, i_j\} \subseteq C \\ i_1 < i_2 < \ldots < i_j}} 2^{i_1 + i_2 + \ldots + i_j - j} \right] \,.
\end{multline}
Now it can be seen that each term in the first summation of \eqref{eq: diffeval} is of order
$$O\left(n^{m} (2+\epsilon)^{(m-j)n}2^{jn}\right)\,,$$ for $j=1,2,\ldots,m$ and $n \to \infty$. Hence, asymptotically in $n$, we can rewrite \eqref{eq: diffeval} as
$$
e_{\Sigma}^m (B) - e_{\Sigma}^m (C) = (2+\epsilon)^{mn} \left[ \binom{|B|}{m} - \binom{|C|}{m} \right] (1 + o(1)) \,,
$$
since $\epsilon > 0$. This clearly contradicts \eqref{eq: eqcon}, and hence we must have $|B| = |C|$.

Next, let $t$ be an integer such that $|B|=|C|=t$ and let $B \coloneqq \{b_1, b_2, \ldots, b_t\}$ and $C \coloneqq \{c_1, c_2, \ldots, c_t\}$, where $b_i, c_i \in [1,n]$ for every $i \in [1,t]$. Then we have:

\begin{multline} \label{eq: reduceval}
\left|e_{\Sigma}^m (B) - e_{\Sigma}^m (C)\right| = \Bigg| (2+\epsilon)^{(m-1)n} \binom{t-1}{m-1} \left( \sum_{i=1}^t 2^{b_i-1} - 2^{c_i-1} \right) + \sum_{j=2}^m (-1)^{j-1} \\ (2+\epsilon)^{(m-j)n} \binom{t-j}{m-j} \left( \sum_{\substack{1 \leq i_1 < i_2 < \ldots < i_j \leq t}} 2^{b_{i_1} + b_{i_2} + \ldots + b_{i_j}-j} - 2^{c_{i_1} + c_{i_2} + \ldots + c_{i_j}-j} \right) \Bigg|\,.
\end{multline}
To conclude the proof, we need to lower bound \eqref{eq: reduceval}. Since $B \neq C$, it is easy to see that
$$
\left|\sum_{i = 1}^t 2^{b_i-1} - 2^{c_i-1} \right| \geq 1
$$
and since each term in the summation over $j$ in \eqref{eq: reduceval} is, as $n \to \infty$, of order
$$O\left(n^{m} (2+\epsilon)^{(m-j)n} 2^{j n} \right)\,,$$
we obtain the following lower bound:
$$
\left|e_{\Sigma}^m (B) - e_{\Sigma}^m (C)\right| > \left|(2+\epsilon)^{(m-1)n} \binom{t-1}{m-1} \right| (1 + o(1)) \,.
$$
The claim now follows from the fact that the right-hand side of the above inequality is greater than $1$ for sufficiently large $n$'s.
\end{proof}

Next, we adapt the idea of Lemma \ref{lem: upper} to work over the integers, taking integer parts.

\begin{Lemma}\label{cor: integers}
Let $\epsilon$ be a positive real number and let $m \geq 2$ be an integer. Then for every $n$ large enough the sequence $\Sigma = (a_1, a_2, \ldots, a_n)$, where
$$a_i = \left\lfloor(2+\epsilon)^n - 2^{i-1}\right\rfloor \mbox{ for } i=1,2,\ldots,n,$$
is $m$-evaluation distinct.
\end{Lemma}

\begin{proof}
Write $a_i = (2+\epsilon)^n - 2^{i-1} - \delta$, where $\delta$ is the fractional part of $(2+\epsilon)^n$.
Proceeding as in the proof of Lemma \ref{lem: upper}, we suppose by contradiction that there exist two distinct subsets $B, C \subseteq [1,n]$ such that
\begin{equation}\label{eq: eqcon2}
|e_{\Sigma}^m (B) - e_{\Sigma}^m (C)| = 0\,.
\end{equation}

We first show that \eqref{eq: eqcon2} implies $|B| = |C|$. Suppose without loss of generality that $|B| > |C|$. Then we have that
\begin{multline}\label{eq: diffeval2}
e_{\Sigma}^m (B) - e_{\Sigma}^m (C) = (2+\epsilon)^{mn} \left[ \binom{|B|}{m} - \binom{|C|}{m} \right] + \sum_{j=1}^m (-1)^j (2+\epsilon)^{(m-j)n} \\ \Bigg[\binom{|B|-j}{m-j} \sum_{\substack{\{i_1, i_2, \ldots, i_j\} \subseteq B \\ i_1 < i_2 < \ldots < i_j}} (2^{i_1-1} + \delta) \cdots (2^{i_j-1} + \delta) - \qquad \qquad \qquad \qquad \qquad \\ \binom{|C|-j}{m-j} \sum_{\substack{\{i_1, i_2, \ldots, i_j\} \subseteq C \\ i_1 < i_2 < \ldots < i_j}} (2^{i_1-1} + \delta) \cdots (2^{i_j-1} + \delta) \Bigg] \,.
\end{multline}
We observe that for fixed $i_1, i_2, \ldots, i_j$ and for $n$ large enough we have that
$$
(2^{i_1-1} + \delta) \cdots (2^{i_j-1} + \delta) \leq (2^n +1)^j \leq e 2^{jn}\,,
$$
since $\delta \leq 1$ and $j \leq m \leq n$.
It follows that each term in the first summation of \eqref{eq: diffeval2} is of order
$$O\left(n^{m} (2+\epsilon)^{(m-j)n}2^{jn}\right)\,,$$ for $j=1,2,\ldots,m$ and $n \to \infty$. Hence, asymptotically in $n$, we can rewrite \eqref{eq: diffeval2} as
$$
e_{\Sigma}^m (B) - e_{\Sigma}^m (C) = (2+\epsilon)^{mn} \left[ \binom{|B|}{m} - \binom{|C|}{m} \right] (1 + o(1)) \,,
$$
since $\epsilon > 0$. This clearly contradicts \eqref{eq: eqcon2}, and hence we must have $|B| = |C|$.

Now, as done in Lemma \ref{lem: upper}, let $t$ be an integer such that $|B|=|C|=t$ and let $B \coloneqq \{b_1, b_2, \ldots, b_t\}$ and $C \coloneqq \{c_1, c_2, \ldots, c_t\}$, where $b_i, c_i \in [1,n]$ for every $i \in [1,t]$. Then we have:

\begin{multline} \label{eq: reduceval2}
\left|e_{\Sigma}^m (B) - e_{\Sigma}^m (C)\right| = \\ \Bigg| (2+\epsilon)^{(m-1)n} \binom{t-1}{m-1} \left( \sum_{i=1}^t \left(2^{b_i-1} + \delta\right) - \left(2^{c_i-1}+\delta\right) \right) + \sum_{j=2}^m (-1)^{j-1} (2+\epsilon)^{(m-j)n} \\ \binom{t-j}{m-j} \left( \sum_{\substack{1 \leq i_1 < i_2 < \ldots < i_j \leq t}} (2^{b_{i_1}-1} + \delta) \cdots (2^{b_{i_j}-1} + \delta) - (2^{c_{i_1}-1} + \delta) \cdots (2^{c_{i_j}-1} + \delta) \right) \Bigg|\,.
\end{multline}

Since $B \neq C$, we have that $\left|\sum_{i = 1}^t 2^{b_i-1} - 2^{c_i-1} \right| \geq 1$. On the other hand, each term in the summation over $j$ in \eqref{eq: reduceval2} is, as $n \to \infty$, of order $O\left(n^{m} (2+\epsilon)^{(m-j)n} 2^{j n} \right)$. Consequently we obtain the following lower bound:
$$
\left|e_{\Sigma}^m (B) - e_{\Sigma}^m (C)\right| > \left|(2+\epsilon)^{(m-1)n} \binom{t-1}{m-1} \right| (1 + o(1)) \,.
$$
Once again, the right-hand side of the above inequality is greater than $1$ for sufficiently large $n$'s, and the proof is complete.
\end{proof}

As a consequence of Corollary \ref{cor: integers}, we obtain the following theorem.

\begin{Theorem}\label{thm: ubdirect}
There exists a sequence $\Sigma = (a_1, a_2, \ldots, a_n)$ of $n$ integers that is $m$-evaluation distinct and $M$-bounded such that
$$
M \leq 2^{n(1 + o(1))}\,,
$$
for $n \to \infty$.
\end{Theorem}
\begin{proof}
 For every $h\in \mathbb{N}_{\ge 1}$ let $\epsilon_h = 2/h$, and let $\Sigma_{n,h} \coloneqq (a_1, a_2, \ldots, a_n)$ with $a_i=\lfloor (2+\epsilon_h)^n-2^{i-1}\rfloor$ for every $i\ge 1$. By Corollary \ref{cor: integers}, there exists a positive integer $n_h$ such that for $n\ge n_h$ the sequence $\Sigma_{n,h}$ is $(2+\epsilon_h)^n$-bounded and $m$-th evaluation distinct. Since $\epsilon_h = 2/h$, we have that $\Sigma_{n,h}$ is $(2^{n + n/h \log_2 e})$-bounded. We may assume, without loss of generality, that the sequence $\{n_h\}_{h\in \N}$ is increasing in $h$. Now the sequences of integers $\Sigma_{n_h,h}$ are $m$-evaluation distinct and $M$-bounded where
$$
M \leq 2^{n_h + \frac{1}{h} n_h \log_2 e}.
$$
Letting $h\to \infty$, one obtains the claim.
\end{proof}

\begin{Remark}
We have stated Theorem \ref{thm: ubdirect} for $\lambda=k=1$ but, clearly, the same bound holds also for any integer $k\geq 1$ and any positive real $\lambda\leq 1$. Similarly, the result of Theorem \ref{thm: upprobm} holds also for $\lambda\leq 1$. Here, assuming $k=1$, the bound given in Theorem \ref{thm: ubprob} improves that of Theorem \ref{thm: ubdirect} for every $\lambda \in [0, 0.11]$ and for every $\lambda < 1/2$ when $k=2$. If instead $k\geq 3$, then the best upper bound is that of Theorem \ref{thm: upprobm} when $\lambda>1/2$ and that of Theorem \ref{thm: ubprob} otherwise.
\end{Remark}

\section*{Acknowledgements}
The first and the third authors were partially supported by INdAM--GNSAGA.

\end{document}